\documentclass{amsproc}

\newtheorem{theorem}{Theorem}[section]

\newtheorem{corollary}[theorem]{Corollary}
\newtheorem{proposition}[theorem]{Proposition}

\theoremstyle{definition}

\theoremstyle{remark}

\numberwithin{equation}{section}

\DeclareMathOperator{\spn}{span}
\DeclareMathOperator{\Real}{Re}
\DeclareMathOperator{\Imag}{Im}
\renewcommand{\Re}{\Real}
\renewcommand{\Im}{\Imag}

\numberwithin{equation}{section}

\begin{document}

\title{Transformations of polynomial ensembles}

\author{Arno B. J. Kuijlaars}
\address{KU Leuven, Department of Mathematics, 
Celestijnenlaan 200B box 2400, 3001 Leuven, Belgium}
\email{arno.kuijlaars@wis.kuleuven.be}

\date{October 6, 2014.}

\dedicatory{Dedicated to Ed Saff  on the occasion of his 70th birthday}


\begin{abstract} 
A polynomial ensemble is a  probability density function for the position
of $n$ real particles of the form 
$\frac{1}{Z_n} \, \prod_{j<k} (x_k-x_j) \, \det \left[ f_k (x_j) \right]_{j,k=1}^n$,
for certain functions $f_1, \ldots, f_n$. Such ensembles appear frequently
as the joint eigenvalue density of random matrices. 
We present a number of  transformations that 
preserve the structure of a polynomial ensemble. These transformations include
the restriction of a Hermitian matrix by removing one row and one column, 
a rank-one modification of a Hermitian matrix, and the extension
of a Hermitian matrix by adding an extra row and column with complex Gaussians.

\end{abstract}

\maketitle



\section{Polynomial ensembles}

A polynomial ensemble is a probability density function on $\mathbb R^n$ of
the form
\begin{equation} \label{f-pol-ensemble} \mathcal P(x_1,\ldots, x_n) = 
	\frac{1}{Z_n} \, \Delta_n(x) \, \det \left[ f_k (x_j) \right]_{j,k=1}^n,
	\end{equation}
where $f_1, \ldots, f_n$ is a given sequence of real-valued functions, 
\[ \Delta_n(x) = \prod_{j < k} (x_k - x_j) = \det [ x_j^{k-1}]_{j,k=1}^n,
\qquad x = (x_1, \ldots, x_n) \in \mathbb R^n, \] 
denotes the Vandermonde determinant, and $Z_n$ is a normalization constant. 
Certain conditions on the functions $f_1, \ldots, f_n$ have to be satisfied to
ensure that \eqref{f-pol-ensemble} is indeed a probability density. For example,
the functions should be linearly independent and the integrals
\[ \int_{-\infty}^{\infty} x^{j-1} f_k(x) dx, \qquad j, k = 1, \ldots, n \]
should be convergent. In addition, \eqref{f-pol-ensemble} has to be non-negative for
all possible choices of $(x_1, \ldots, x_n \in \mathbb R^n$. The probability density \eqref{f-pol-ensemble}
only depends on the linear span of the functions $f_1, \ldots, f_n$, as one can see
by applying column transformations to the second determinant in \eqref{f-pol-ensemble}. 

A special case arises if $f_k(x) = x^{k-1} w(x)$, for $k=1, \ldots n$, with
$w$ an integrable non-negative function on $\mathbb R$ such that the moments
up to order $2n-2$ exist. In that
case \eqref{f-pol-ensemble} can be written as
\begin{equation} \label{opol-ensemble} \mathcal P(x_1,\ldots, x_n) = 
	\frac{1}{Z_n} \, \Delta_n^2(x) \, \prod_{j=1}^n w(x_j),
		\qquad x = (x_1, \ldots, x_n) \in \mathbb R^n,
	\end{equation}
and \eqref{opol-ensemble} is known as an orthogonal polynomial ensemble \cite{Konig},
as the analysis of \eqref{opol-ensemble} relies on the polynomials that
are orthogonal with respect to $w$. The ensembles \eqref{opol-ensemble} arise as the
joint probability density of eigenvalues of unitary invariant ensembles of
Hermitian random matrices \cite{Deift}.
 The polynomial ensembles \eqref{f-pol-ensemble}
also include the multiple orthogonal polynomials ensembles, see \cite{Kui},
where also more examples from random matrix theory are given.

On the other hand, we have that \eqref{f-pol-ensemble} is a special case
of the more general class of biorthogonal ensembles, see \cite{Bor,ClRo},
\begin{equation} \label{bio-ensemble} \mathcal P(x_1,\ldots, x_n) = 
	\frac{1}{Z_n} \, 	\det\left[ g_k(x_j) \right]_{j,k=1}^n \,  \det \left[ f_k (x_j) \right]_{j,k=1}^n,
	\end{equation}
which involves two sequences  $g_1, \ldots, g_n$ and $f_1, \ldots, f_n$  of given functions.
It is known that \eqref{bio-ensemble} is determinantal \cite{Bor}, which means
that there exists a kernel $K_n : \mathbb R \times \mathbb R \to \mathbb R$ such
that
\[ \mathcal P(x_1,\ldots, x_n) = \frac{1}{n!} \det\left[ K_n(x_j,x_k) \right]_{j,k=1}^n \]
and such that for every $m = 1, \ldots, n-1$,
\[ \int_{\mathbb R^{n-m}} \mathcal P(x_1, \ldots, x_n) d x_{m+1} \cdots dx_n
	=  \frac{(n-m)!}{n!} \det \left[ K_n(x_j, x_k) \right]_{j,k=1}^m. \]
The correlation kernel $K_n$ has the form
\[ K_n(x,y) = \sum_{k=1}^n \psi_k(x) \phi_k(y) \]
where $\spn \{ \psi_1, \ldots, \psi_n \} = \spn \{g_1, \ldots, g_n \}$,
$\spn \{ \phi_1, \ldots, \phi_n \} = \spn \{f_1, \ldots, f_n \}$ and  
\[ \int_{-\infty}^{\infty} \psi_j(x) \phi_k(x) dx = \delta_{j,k} \qquad j, k =1,\ldots, n. \]
In the case \eqref{f-pol-ensemble}  we may write
\[ K_n(x,y) = \sum_{j=0}^{n-1} P_j(x) Q_j(y) \]
where $P_j$ is a monic polynomial of degree $j$ for $j=0, \ldots, n-1$,
the dual functions $Q_0, \ldots, Q_{n-1}$ are in the linear span of $f_1, \ldots, f_n$, and
the biorthogonality condition 
\begin{equation} \label{biorthogonal} 
	\int_{-\infty}^{\infty} P_j(x) Q_k(x) dx = \delta_{j,k}, \qquad j,k=0, \ldots, n-1 
	\end{equation}
is satisfied. 
In this case we can also consider the monic polynomial $P_n$ of degree $n$ such that 
\eqref{biorthogonal} also holds for $j=n$. This polynomial is given by
\[ P_n(x) = \mathbb E \left[ \prod_{j=1}^n (x-x_j) \right] \]
where the averaging is over $(x_1, \ldots, x_n)$ in the polynomial ensemble \eqref{f-pol-ensemble}.
If the points $x_1, \ldots, x_n$ come from eigenvalues of a random matrix, then $P_n$ is
the average characteristic polynomial.

\section{Known transformations}

This paper discusses a number of  transformations that preserve the
structure of a polynomial ensemble. These transformations come from random matrix theory,
and the typical setting is the following. We assume that $X$ is a random matrix whose
eigenvalues (or squared singular values) are distributed according to a polynomial
ensemble \eqref{f-pol-ensemble}. Then we perform a certain transformation to obtain
from $X$ a new random matrix $Y$, and the result is that the eigenvalues (or squared
singular values) of $Y$ are again a polynomial ensemble.

\subsection{Product with Ginibre matrix}
The first example of such a transformation comes from recent work of the author
with Dries Stivigny \cite{KuSt}.  
It deals with the squared singular values of rectangular matrices. Recall that the
squared singular values of a rectangular complex matrix $X$ are the eigenvalues of $X^*X$. 
The transformation on $X$ is multiplication by a complex Ginibre matrix, where
a complex Ginibre matrix is a random matrix whose entries are independent standard
complex Gaussians.

\begin{theorem} \label{GinibreTransformation}
Let $n, l, \nu$ be non-negative integers with $1 \leq n \leq l$.
Let $G$ be an $(n+\nu) \times l$ complex Ginibre matrix, and let
$X$ be a random matrix of size $l \times n$, independent of $G$,
such that the squared singular values $x_1, \ldots, x_n$ are a polynomial ensemble \eqref{f-pol-ensemble}
for certain functions $f_1, \ldots, f_n$ defined on $[0,\infty)$.
Then the squared singular values $y_1, \ldots, y_n$ of $Y = GX$ are a polynomial ensemble 
\begin{equation} \label{g-pol-ensemble} 
	\frac{1}{\tilde{Z}_n} \Delta_n(y) \, \det \left[ g_{k}(y_j) \right]_{j,k=1}^n, \qquad \text{ all } y_j > 0, 
	\end{equation}
where 
\begin{equation} \label{gk-definition} 
	g_k(y) = \int_0^{\infty} x^{\nu} e^{-x} f_k \left( \frac{y}{x} \right) \frac{dx}{x}, \qquad y > 0.
	\end{equation}
\end{theorem}
\begin{proof} See \cite{KuSt}, where the proof is based on ideas taken from \cite{AIK,AKW}.
\end{proof}
Note that $g_k$ in \eqref{gk-definition} is the Mellin convolution of $x \mapsto x^{\nu} e^{-x}$ with $f_k$.

Theorem \ref{GinibreTransformation} can be applied repeatedly and it follows that the multiplication with
any number of complex Ginibre matrices preserves the structure of a polynomial ensemble
for the squared singular values.

Theorem \ref{GinibreTransformation} was inspired by earlier results by 
Akemann et al.\ \cite{AIK, AKW} on products of random matrices. In these papers the authors considered  products of 
complex Ginibre matrices (that is, $X$ is also a complex Ginibre matrix) and 
they obtained the structure \eqref{g-pol-ensemble}--\eqref{gk-definition},
where in this case the functions $g_k$ in \eqref{g-pol-ensemble} are expressed as Meijer G-functions.
This result has since then been used in \cite{KuZh,LWZ} to determine the large $n$ scaling limit
of the correlation kernel at the hard edge, and in \cite{ABK} to calculate the Lyaponov 
exponents as the number of matrices in the product tends to infinity. 
See also \cite{For4,Strahov, Zhang} for other recent results on singular values of products
of random matrices.

\subsection{Product with a truncated unitary matrix}

Theorem \ref{GinibreTransformation} has an extension to a product
with a truncated unitary matrix. A truncation  $T$ of a matrix $U$ is a 
principal submatrix of $U$. We
assume that $U$ is a Haar distributed random unitary matrix and then $T$ is also a random matrix.

\begin{theorem} \label{TruncationTransformation}
Let $n, m, l, \nu$ be non-negative integers with $n \leq l \leq m $ and $m \geq n + \nu +1$. 
Let $T$ be an $(n+\nu) \times l$ truncation of a Haar distributed
unitary matrix $U$ of size $m \times m$.
Let $X$ be a random matrix of size $l \times n$, independent of $U$,
such that the squared singular values $x_1, \ldots, x_n$ of $X$ are a polynomial ensemble \eqref{f-pol-ensemble} 
for certain functions $f_1, \ldots, f_n$ defined on $[0,\infty)$.
Then the squared singular values $y_1, \ldots, y_n$ of $Y = TX$ are a polynomial ensemble 
\begin{equation} \label{g-pol-ensemble2} 
	\frac{1}{\tilde{Z}_n} \, \Delta_n(y) \, \det \left[ g_{k}(y_j) \right]_{j,k=1}^n, \qquad \text{ all } y_j > 0, 
	\end{equation}
where 
\begin{equation} \label{gk-definition2} 
	g_k(y) = \int_0^1 x^{\nu} (1-x)^{m-n-\nu-1} f_k \left( \frac{y}{x} \right) \frac{dx}{x}, \qquad y > 0.
	\end{equation}
\end{theorem}
\begin{proof}
See \cite{KKS}.
\end{proof}

If we let $m \to \infty$ in Theorem \ref{TruncationTransformation}, then $\sqrt{m} \, T$ tends
in distribution to a complex Ginibre matrix. Also $(1- \tfrac{x}{m})^{m-n-\nu-1}$ tends 
to $e^{-x}$ as $ m \to \infty$. In this way Theorem~\ref{GinibreTransformation} can be obtained
as a limiting case of Theorem \ref{TruncationTransformation}.

Theorems \ref{GinibreTransformation} and \ref{TruncationTransformation} can be used repeatedly
and it follows that the squared singular values of a product of
any number of Ginibre
matrices with any number of truncated unitary matrices are a polynomial ensemble.

\subsection{Overview of the rest of the paper}
Inspired by these results we give an overview of other transformations that preserve
polynomial ensembles. The transformations are based on known random
matrix theory calculations, see \cite{For,FR}, and our aim here is to 
emphasize the interpretation as a transformation of polynomial ensembles. 

The first such transformation comes from matrix restrictions.
Here we are working with a Hermitian matrix $X$ and we remove one row and one column
to obtain $Y$. If $X$ is random with eigenvalues that are distributed as a polynomial ensemble 
then the eigenvalues of $Y$ are also distributed as a polynomial ensemble. This is our
first result, see Theorem \ref{SubmatrixRestriction}. The proof relies on a fundamental
result of Baryshnikov \cite{Bar}, see Theorem \ref{Baryshnikov} below.

Then we extend this to the situation where $X$ is a positive semidefinite matrix with 
a fixed number of zero eigenvalues. 
Again we find that matrix restriction for random matrices of this type leads to a transformation result for
polynomial ensembles, see  Theorem \ref{PosDefRestriction}.
Interestingly enough, we can make a connection with the product with a truncated
unitary matrix, as we find in this way an alternative proof for Theorem \ref{TruncationTransformation}.

In Section \ref{sec:rankone} we consider a transformation from $X$ to $Y = X + vv^*$ where
$X$ is Hermitian, and $v$ is a column vector of independent complex Gaussian entries.
This rank-one modification is also a transformation of polynomial ensembles as we show 
in Proposition \ref{SmallRankModification}. The argument is based on a result of \cite{FN}.

Finally, in Section \ref{sec:extension} we consider a transformation where we extend
the Hermitian matrix $X$ by adding an extra column $v$ with independent complex Gaussians, and
an extra row $\begin{pmatrix} v^* & c \end{pmatrix}$ consisting of $v^*$ and a
real  number $c$ that has a real normal distribution. Under appropriate conditions
on the variances, we again find a transformation of  polynomial ensembles, 
see Proposition \ref{HermitianExtension}. This is based on \cite{AvMW,For2}.


\section{Matrix restrictions} \label{sec:restriction}

Let $X$ be an $n \times n$ Hermitian matrix with distinct eigenvalues $x_1 < x_2 < \cdots < x_n$.
Let $U$ be a Haar distributed unitary matrix of size $n \times n$ and let $Y$ be the $(n-1) \times (n-1)$
principal submatrix of $U X U^*$ with eigenvalues $y_1 \leq y_2 \leq \cdots \leq y_{n-1}$. With probability
one we have strict interlacing of eigenvalues
\begin{equation} \label{xy-interlacing}
	x_1 < y_1 < x_2 < y_2 < \cdots < y_{n-1} < x_n. 
	\end{equation}
The following theorem is due to Baryshnikov (reformulation of \cite[Proposition 4.2]{Bar}).

\begin{theorem} \label{Baryshnikov}
If $X$ and $Y$ are as above, then the (random) eigenvalues $y_1, \ldots, y_{n-1}$ of $Y$ have 
the joint density 
\begin{equation} \label{y-density} 
	(n-1)! \, \frac{\Delta_{n-1}(y)}{\Delta_n(x)} \end{equation}
on the subset of $\mathbb R^{n-1}$ defined by the inequalities \eqref{xy-interlacing}.
\end{theorem}
	
The interlacing condition is expressed by the determinant
\begin{equation} \label{interlacing-det} 
	\det \left[ \chi_{x_k \leq y_j} \right]_{j,k=1}^n, \qquad \chi_{x \leq y} = \begin{cases} 1 & \text{ if } x \leq y \\
	0 & \text{ otherwise}, \end{cases} \end{equation}
with $y_n := +\infty$. Indeed, for all mutually distinct values $x_k$ and $y_j$, 
the determinant in \eqref{interlacing-det} is $1$ if and only if the interlacing condition holds 
and it is zero otherwise. The determinant in \eqref{interlacing-det} has all ones in the
last row. We can reduce it to an $(n-1) \times (n-1)$ determinant by subtracting the last column from
every other column, and expanding along the last row. This results in the determinant   
$\det \left[ \chi_{x_k \leq y_j < x_n} \right]_{j,k=1}^{n-1}$. 
It means that the density \eqref{y-density} can be written as  
\begin{multline} \label{y-density2}
	\frac{\Delta_{n-1}(y)}{\Delta_n(x)} \, \det\left[ \chi_{x_k \leq y_j} \right]_{j,k=1}^n 
	=\frac{\Delta_{n-1}(y)}{\Delta_n(x)} \,  \det \left[ \chi_{x_k \leq y_j < x_n} \right]_{j,k=1}^{n-1},  \\
	\qquad y = (y_1, \ldots, y_{n-1}) \in \mathbb R^{n-1}, \quad y_n := +\infty,
	\end{multline}
where \eqref{y-density2} is  considered as a probability density on the unordered eigenvalues of $Y$, i.e.,
as a probability density on $\mathbb R^{n-1}$. This accounts for the disappearance of the factor $(n-1)!$ 
from \eqref{y-density}.
Note that \eqref{y-density2} is a polynomial
ensemble with functions $y \mapsto \chi_{x_k \leq y < x_n}$, for $k=1,2, \ldots, n-1$.

Let us now assume that $X$ is random, independent of $U$, and that the eigenvalues of $X$ are a polynomial ensemble. 
Then the eigenvalues of $Y$ are again a polynomial ensemble. For this it is important that 
the normalization constant  $\frac{1}{\Delta_n(x)}$ in \eqref{y-density2} depends on $X$ via the 
Vandermonde determinant $\Delta_n(x)$ in the denominator. We also need the  Andreief identity,
 see \cite[Chapter 3]{DeGi},
\begin{multline} \label{Andreief-identity} 
	\int_{\mathcal X^n} \det \left[ \phi_k(x_j) \right]_{j,k=1}^n \, 
	\det\left[ \psi_k(x_j) \right]_{j,k=1}^n \, d\mu(x_1) \cdots d\mu(x_n) \\
 = n! \det \left[ \int_{\mathcal X} \phi_j(x) \psi_k(x) d\mu(x) \right]_{j,k=1}^n 
\end{multline}
where $\mu$ is  a measure on a space $\mathcal X$, $n \in \mathbb N$, and $\phi_1, \ldots, \phi_n$, $\psi_1, \ldots, \psi_n$
are arbitrary functions such that the integrals converge.
The result is the following.

\begin{theorem} \label{SubmatrixRestriction}
Suppose that $X$ is a random $n \times n$ Hermitian matrix whose eigenvalues
are a polynomial ensemble \eqref{f-pol-ensemble} with certain functions $f_1, \ldots, f_n$. 
Let $Y$ be the principal submatrix of $U X U^*$ of size $(n-1) \times (n-1)$, 
where $U$ is a Haar distributed  unitary matrix, independent of $X$.
Then the eigenvalues $y_1, \ldots, y_{n-1}$ of $Y$ are a polynomial ensemble
\begin{equation} \label{g-pol-ensemble3} 
	\frac{1}{\tilde{Z}_{n-1}} \, \Delta_{n-1}(y) \, \det \left[ g_{k}(y_j) \right]_{j,k=1}^{n-1}, 
	\end{equation}
with
\begin{equation} \label{gk-definitions3} 
	g_k(y) = \int_{-\infty}^y \tilde{f}_k(x) \, dx, \qquad k=1, \ldots, n-1,
	\end{equation}
where $\tilde{f}_1, \ldots, \tilde{f}_{n-1}$ are a basis for the vector space
\begin{equation} \label{subspace} 
	\{ f = \sum_{j=1}^n c_j f_j \mid c_1, \ldots, c_n \in \mathbb R, \, \int_{-\infty}^{\infty} f(x) \, dx = 0 \}. 
	\end{equation}
\end{theorem}

\begin{proof}
From  \eqref{y-density2} it follows after averaging over the polynomial ensemble \eqref{f-pol-ensemble} 
that the eigenvalues of $Y$ have joint density
\[   \frac{1}{n! Z_n} \Delta_{n-1}(y) 	\int_{\mathbb R^n}
	\det\left[ \chi_{x_k \leq y_j} \right]_{j,k=1}^n \det \left[ f_k(x_j) \right]_{j,k=1}^n \, dx_1 \cdots dx_n \]
	with $y_n := +\infty$.
Because of \eqref{Andreief-identity} we find	that this is
\begin{equation} \label{y-density3}   
	\frac{1}{Z_n} \Delta_{n-1}(y) \det \left[	\int_{-\infty}^{y_j}  f_k(x) dx \right]_{j,k=1}^n. 
	\end{equation}
Since $y_n = +\infty$, the last row of the second determinant in \eqref{y-density3} contains the constants
$\int_{-\infty}^{\infty} f_k(x) dx$. Change from $f_1, \ldots, f_n$ to another basis $\tilde{f}_1, \ldots \tilde{f}_n$
where $\tilde{f}_1, \ldots, \tilde{f}_{n-1}$ belong to the subspace \eqref{subspace}. Then
after performing suitable column operations we obtain the density
\[   \frac{1}{\tilde{Z}_{n-1}} \Delta_{n-1}(y) \det \left[	\int_{-\infty}^{y_j}  \tilde{f}_k(x) dx \right]_{j,k=1}^n. \]
Then $\int_{-\infty}^{\infty} \tilde{f}_n(x) dx \neq 0$, since otherwise the full last row 
in the determinant would be zero.
By expanding the determinant along the last row we obtain \eqref{g-pol-ensemble3} with
functions \eqref{gk-definitions3} and a possibly different constant $\tilde{Z}_{n-1}$.
\end{proof}

An analogous result holds for singular values.

\begin{corollary} \label{SubmatrixRestriction2}
Suppose that $X$ is random $(n+\nu) \times n$ matrix whose squared singular values 
are a polynomial ensemble \eqref{f-pol-ensemble}. Let $Y$ be the $(n+ \nu) \times (n-1)$
principal submatrix of $X U$ where $U$ is Haar distributed unitary matrix, independent of $X$.
Then the squared singular values $y_1, \ldots, y_{n-1}$ of $Y$ are a polynomial ensemble
\begin{equation} \label{g-pol-ensemble4} 
	\frac{1}{\tilde{Z}_{n-1}} \, \Delta_{n-1}(y) \, \det \left[ g_{k}(y_j) \right]_{j,k=1}^{n-1},  \qquad \text{ all } y_j > 0,
	\end{equation}
with
\begin{equation} \label{gk-definitions4} 
	g_k(y) = \int_{0}^y \tilde{f}_k(x) dx, \qquad k=1, \ldots, n-1 
	\end{equation}
where $\tilde{f}_1, \ldots, \tilde{f}_{n-1}$ are a basis for the vector space
\begin{equation} \label{subspace2} 
	\{ f = \sum_{j=1}^n c_j f_j \mid c_1, \ldots, c_n \in \mathbb R, \, \int_{0}^{\infty} f(x) dx = 0 \}. 
	\end{equation}
\end{corollary}
\begin{proof}
We can apply Theorem \ref{SubmatrixRestriction} since
$Y^*Y$ is the principal submatrix of size $(n-1) \times (n-1)$ of $X^*X$.
The integration in \eqref{gk-definitions4} and \eqref{subspace2} starts at $0$ since the functions are defined
for $x \geq 0$ only.
\end{proof}

\section{Restrictions of positive semidefinite matrices} \label{sec:posdefrestriction}

The following is a variation on Theorem \ref{Baryshnikov}. It can  also be obtained
as a special case of \cite[Corollary 1]{FR}.

\begin{proposition} \label{varBaryshnikov}
Let $m \geq n +1$ and let $X$ be an $m \times m$ positive semidefinite  Hermitian matrix with 
$n$ simple non-zero eigenvalues $ 0 < x_1 < x_2 < \cdots < x_n$
and an eigenvalue $0$ of multiplicity $m-n \geq 1$.
Let $Y$ be the $(m-1) \times (m-1)$ principal submatrix of $U X U^*$ where $U$ is a Haar distributed 
unitary matrix of size $m \times m$.
Then with probability one, $Y$ has exactly $n$ non-zero eigenvalues $0 < y_1 < y_2 < \cdots < y_n$ 
that satisfy the inequalities
\begin{equation} \label{xy-inequalities} 
	0  < y_1 <  x_1 < y_2 < x_2 < \cdots < x_{n-1} < y_n < x_n,  
	\end{equation}
	and these non-zero eigenvalues have the joint density
\begin{equation} \label{y-density4} 
	\frac{(m-1)!}{(m-n-1)!} \, \prod_{k=1}^n \frac{y_k^{m-n-1}}{x_k^{m-n}}  \frac{ \Delta_n(y)}{\Delta_n(x)}
	\end{equation}
restricted to the  subset of $y = (y_1, \ldots, y_n) \in \mathbb R^{n}$ 
defined by the inequalities \eqref{xy-inequalities}.
\end{proposition}

\begin{proof} For $m=n+1$ this follows immediately from Theorem \ref{Baryshnikov}
and so we assume in the proof that $m \geq n+2$.
We approximate $X$ by a matrix $A$ with eigenvalues $a_1 < \cdots < a_{m-n} < x_1 < \cdots < x_n$
with $a_j$'s close to zero. Let $B$ be the principal submatrix of 
$U A U^*$ of size $(m-1) \times (m-1)$, which with probability one
 has distinct eigenvalues $b_1 < \cdots < b_{m-n-1} < y_1 < \cdots < y_n$ 
that interlace with the eigenvalues of $A$. By Theorem \ref{Baryshnikov} the  joint density
of these eigenvalues is  
\begin{equation} \label{ytilde-density} 
 (m-1)! \frac{ \Delta_{m-n-1}(b) \prod_{j,k} (y_k - b_j)}{ \Delta_{m-n}(a) \prod_{j,k} (x_k-a_j)}
	\frac{\Delta_n(y)}{\Delta_n(x)} \end{equation}
on the subset of $\mathbb R^m$ given by the interlacing relations.
The induced density on $y_1, \ldots, y_n$ is
\begin{equation} \label{ytilde-density2} 
	(m-1)! \left( \int_{a_1}^{a_2} db_1 \cdots \int_{a_{m-n-1}}^{a_{m-n}} db_{m-n-1} 
	\frac{ \Delta_{m-n-1}(b) \prod_{j,k} (y_k - b_j)}{\Delta_{m-n}(a) \prod_{j,k} (x_k-a_j)}
	\right) \frac{\Delta_n(y)}{\Delta_n(x)}. 
	\end{equation}
In the limit where all $a_j \to 0$, $j=1, \ldots, m-n$, we also have $b_j \to 0$, $j=1, \ldots, m-n-1$.
Then the factors	$\prod_{j,k} (y_k - b_j)$ and $\prod_{j,k} (x_k - a_j)$ in \eqref{ytilde-density2}
tend to $\prod_k y_k^{m-n-1}$ and $\prod_k x_k^{m-n}$, respectively.
The resulting $m-n-1$ fold integral can be evaluated as
\begin{equation} \label{ab-integral} 
	\int_{a_1}^{a_2} db_1 \ldots \int_{a_{m-n-1}}^{a_{m-n}} db_{m-n-1} 
	\frac{ \Delta_{m-n-1}(b)}{\Delta_{m-n}(a)}  = \frac{1}{(m-n-1)!} 
	\end{equation}
and this does not depend on $a_1, \ldots a_{m-n}$.
The result is the joint density \eqref{y-density4} for the non-zero eigenvalues of $Y$.
\end{proof}

The inequalities  \eqref{xy-inequalities} are encoded by the determinant
\[ \det \left[ \chi_{0 < y_j < x_k} \right]_{j,k=1}^n \]
which for strictly increasing $y_1 < y_2 < \cdots < y_n$ is $1$ if the interlacing
\eqref{xy-inequalities} holds and $0$ otherwise.
Then \eqref{y-density4} can be alternatively written as
\begin{equation} \label{y-density5} 
	\frac{(m-1)!}{n! (m-n-1)!} \, \prod_{k=1}^n \frac{y_k^{m-n-1}}{x_k^{m-n}}  \frac{ \Delta_n(y)}{\Delta_n(x)}
		\det \left[ \chi_{0 < y_j < x_k} \right]_{j,k=1}^n 
	\end{equation}
	which is now considered as a density on $[0,\infty)^n$ for unordered eigenvalues. Note that \eqref{y-density5} is
	a polynomial ensemble on $[0,\infty)$ with functions $y \mapsto y^{m-n-1} \chi_{0 < y < x_k}$ for $k =1, \ldots, n$.

\begin{theorem} \label{PosDefRestriction} Let $n \leq m-1$ and $\nu \leq m-n-1$ be positive integers.
Let $X$ be a random positive semidefinite Hermitian matrix of size $m\times m$ with a zero eigenvalue of 
multiplicity $m-n \geq 1$ and non-zero eigenvalues $x_1, \ldots, x_n$ that
are a polynomial ensemble \eqref{f-pol-ensemble} for certain functions $f_1, \ldots, f_n$ on $[0, \infty)$.
Let $Y$ be the principal submatrix of $U XU^*$ of size $(n+\nu) \times (n+\nu)$, where $U$ is a Haar distributed
unitary matrix, independent of $X$. Then, with probability one,
$Y$ has exactly $n$ non-zero eigenvalues $y_1, \ldots, y_n$, and these non-zero eigenvalues
are a polynomial ensemble 
\begin{equation} \label{g-pol-ensemble5} 
	\frac{1}{\tilde{Z}_{n}}  \Delta_n(y) \, \det\left[ g_k(y_j) \right]_{j,k=1}^n \qquad \text{all }  y_j > 0, 
	\end{equation}
where
\begin{equation} \label{gk-definition5} 
	g_k(y) =   \int_0^1 x^{\nu} (1-x)^{m-n-\nu-1} f_k \left( \frac{y}{x} \right) \frac{dx}{x}.
	\end{equation}
\end{theorem}

\begin{proof}
We first assume that $\nu = m-n-1$. Then $Y$ is obtained from $UXU^*$ by removing one row and
column and we can apply Proposition \ref{varBaryshnikov} and in particular its reformulation
in \eqref{y-density5}. Averaging \eqref{y-density5} over the polynomial ensemble \eqref{f-pol-ensemble}
we obtain the joint density
\[ \frac{1}{\tilde{Z}_{n}} \Delta_n(y) \int_{[0,\infty)^n} 
	\prod_{k=1}^n \frac{y_k^{m-n-1}}{x_k^{m-n}}  \det \left[ f_k(x_j) \right]_{j,k=1}^n \, 
		\det \left[ \chi_{0 < y_j < x_k} \right]_{j,k=1}^n \, dx_1 \cdots dx_n, \]
		for a certain constant $\tilde{Z}_n$ (which also depends on $m$).
By the Andreief identity \eqref{Andreief-identity} this leads to
\[ \frac{1}{\tilde{Z}_n} \Delta_n(y) \det \left[ \int_0^{\infty} \frac{y_j^{m-n-1}}{x^{m-n}} f_k(x) \chi_{0 < y_j < x} \, dx \right] \]
with a new constant $\tilde{Z}_n$.
This is a polynomial ensemble \eqref{g-pol-ensemble5} with functions
\begin{align} \nonumber 
	g_k(y) & = \int_0^{\infty}  \frac{y^{m-n-1}}{x^{m-n}} f_k(x) \chi_{0 < y < x} \, dx \\
	 & = \int_y^{\infty} \frac{y^{m-n-1}}{x^{m-n}} f_k(x) \, dx, \qquad y > 0. \label{gk-definitions6}
	\end{align}
The substitution $x \mapsto \frac{y}{x}$ in \eqref{gk-definitions6} leads to
the expression \eqref{gk-definition5} with $\nu = m-n-1$.
This is the Mellin convolution of $f_k$ with the function $\chi_{m-n-1}$ where we define
\[ \chi_k : [0, \infty) \to \mathbb R : \, x \mapsto x^{k} \chi_{0 < x < 1}. \]
Thus $g_k = \chi_{m-n-1} \ast f_k$, if $\nu = m-n-1$,
where $\ast$ is used here for the Mellin convolution
\[ (f \ast g)(y) = \int_0^{\infty} f(x) g\left(\frac{y}{x}\right) \frac{dx}{x}. \]

For general $\nu \leq m-n-1$ we can use the above argument repeatedly, and we
find a polynomial ensemble \eqref{g-pol-ensemble5} with functions $g_k$ 
that are iterated Mellin convolutions of the functions $f_k$, namely
\[ g_k = \chi_{\nu} \ast \chi_{\nu+1} \ast \cdots \ast \chi_{m-n-1} \ast f_k, \qquad k =1, \ldots, n. \]
It is easy to calculate that
\[ \left(\chi_{\nu} \ast \chi_{\nu+1} \ast \cdots \ast \chi_{m-n-1} \right)(x) = x^{\nu} (1-x)^{m-n-\nu-1},
	\qquad 0 < x < 1, \]
and thus we obtain the formula \eqref{gk-definition5} for the functions $g_k$.
\end{proof}

An attentive reader may have noticed that the formula for $g_k$
in \eqref{gk-definition5} coincides with the one appearing in  \eqref{gk-definition2} 
in Theorem \ref{TruncationTransformation}.
This is no coincidence since we can use Theorem \ref{PosDefRestriction} 
to give an alternative proof of Theorem \ref{TruncationTransformation}.

\begin{proof}[Proof of Theorem \ref{TruncationTransformation}]

Let $X$ be an $l \times n$ matrix, and put
\[ \tilde{X} = \begin{pmatrix} XX^* & 0 \\ 0 & 0 \end{pmatrix} \]
which is an $m \times m$ matrix with $m-l$ rows and columns containing only zeros.
It is clear that the squared singular values of $X$ are equal to non-zero eigenvalues of $\tilde{X}$.

Also if $U$ is a unitary matrix of size $m \times m$ and $T$ is its left upper block of size $(n+\nu) \times l$
then 
\[ U \tilde{X} U^* =  \begin{pmatrix} T X X^* T^* & * \\ * & * \end{pmatrix} \]
where $*$ denotes a certain unspecified entry, whose value is not important for us. In other words,
$TX X^* T^*$ is equal to the principal submatrix of $U \tilde{X} U^*$ of size $(n + \nu) \times (n +\nu)$.
Assuming $U$ is Haar distributed over the unitary group, and the squared singular
values of $X$ are a polynomial ensemble \eqref{f-pol-ensemble}, independent of $U$, we then find
from Theorem \ref{PosDefRestriction} that the non-zero eigenvalues of $TXX^* T^*$ are 
a polynomial ensemble \eqref{g-pol-ensemble} with functions \eqref{gk-definition5}.
The non-zero eigenvalues of $TXX^* T^*$ are the same as the squared singular
values of $TX$ and Theorem \ref{TruncationTransformation} follows.
\end{proof}

\section{Rank one modification} \label{sec:rankone}

Let $X$ be a Hermitian $n \times n$ matrix with eigenvalues $x_1 < x_2 < \cdots < x_n$.
We take $Y = X + v v^*$ where $v$ is a vector of length $n$. Then the eigenvalues $y_j$
of $Y$ interlace with those of $X$, as follows from the Courant-Fischer Theorem, see 
e.g.~\cite[chapter 7.5]{Meyer}. 
We let $v  = (v_1, \ldots, v_n)^t$ be a vector of independent complex random variables whose
real and imaginary parts are independent and have a $N(0,1/2)$ distribution.
Then the distribution of the eigenvalues of $Y$ is given in \cite[Appendix E]{FN}
as
\begin{equation} \label{y-density6} 
	\prod_{j=1}^n e^{-(y_j-x_j)} \, \frac{\Delta_n(y)}{\Delta_n(x)} 
	\end{equation}
on the subset of $\mathbb R^n$ given by the interlacing conditions
\begin{equation} \label{xy-inequalities2} 
	x_1 < y_1 < x_2 < \cdots < x_n < y_n.
	\end{equation}
The following result is an immediate consequence.

\begin{proposition} \label{SmallRankModification}
Let $\Re v_j$, $\Im v_j$, for $j=1, \ldots, n$ be mutually independent 
normal random variables with mean zero and variance $1/2$. Let $X$ be a random 
Hermitian matrix of size $n \times n$, independent of $v = (v_1, \ldots, v_n)^t$, 
whose eigenvalues are a polynomial ensemble \eqref{f-pol-ensemble}
with certain functions $f_1, \ldots, f_n$.
Then the eigenvalues $y_1, \ldots, y_{n}$ of $Y = X + v v^*$ are a polynomial ensemble
\begin{equation} \label{g-pol-ensemble7} 
	\frac{1}{\tilde{Z}_{n}} \, \Delta_{n}(y) \, \det \left[ g_{k}(y_j) \right]_{j,k=1}^{n}, 
	\end{equation}
where
\begin{equation}  \label{gk-smallrank} 
	g_k(y) = \int_0^{\infty} e^{-x} f_k(y-x) \, dx, \qquad k=1, \ldots, n.
		\end{equation}
\end{proposition}
Thus $g_k$ is the convolution of $f_k$ with $x \mapsto e^{-x} \chi_{x \geq 0}$.
\begin{proof}
The interlacing \eqref{xy-inequalities2} is encoded by a determinant, and it follows that \eqref{y-density6} is 
a polynomial ensemble 
\begin{equation} \label{y-density7} 
\frac{1}{n!} \prod_{j=1}^n e^{-(y_j-x_j)} \frac{\Delta_n(y)}{\Delta_n(x)} \det \left[ \chi_{x_k < y_j} \right]_{j,k=1}^n,
\end{equation}
where now we disregard the ordering of the $y_j$'s and  consider \eqref{y-density7} as a probability density on $\mathbb R^n$. 

We average over $x_1, \ldots, x_n$ distributed as in \eqref{f-pol-ensemble}. 
By Andreief's identity \eqref{Andreief-identity}, we obtain for the density of the eigenvalues of $Y$
\[ \frac{1}{\tilde{Z}_n}  \Delta_n(y) \det \left[ \int_{-\infty}^{y_j} e^{-(y_j-x)} f_k(x) \, dx \right]_{j,k=1}^n. \]
Changing variables  $x \mapsto y_j-x$ in the integral in the determinant, we arrive at
\eqref{g-pol-ensemble7} with functions \eqref{gk-smallrank}.
\end{proof}

Here is a variation on the same theme.

\begin{proposition} \label{SmallRankModification2} 
Let $n, \nu \geq 1$. Let $\Re v_j$, $\Im v_j$, for $j=1, \ldots, n+\nu$ be mutually independent 
normal random variables with mean zero and variance $1/2$. Let $X$ be a random 
$(n+\nu) \times (n + \nu)$ positive semidefinite Hermitian matrix, independent of $v_1, \ldots, v_n$, 
with exactly $n$ positive eigenvalues $x_1, \ldots, x_n$ that are a polynomial ensemble \eqref{f-pol-ensemble}
with certain functions $f_1, \ldots, f_n$ on $[0,\infty)$.
Then, almost surely, $Y = X + vv^*$ has an eigenvalue zero of multiplicity $\nu-1$ and $n+1$
positive eigenvalues $y_1, \ldots, y_{n+1}$ that are a polynomial ensemble
\begin{equation} \label{g-pol-ensemble8} 
	\frac{1}{\tilde{Z}_{n+1}} \, \Delta_{n+1}(y) \, \det \left[ g_{k}(y_j) \right]_{j,k=1}^{n+1}, 
		\qquad \text{all } y_j > 0,
	\end{equation}
with functions
\begin{equation}  \label{gk-smallrank2} 
\begin{aligned}
	g_{1}(y) & = y^{\nu-1} e^{-y}, \\
	g_{k+1}(y) & = y^{\nu-1} e^{-y} \int_{c}^y x^{-\nu} e^{x} f_{k}(x) dx, \qquad k=1, \ldots, n, 
		\end{aligned}
		\end{equation}
for $y > 0$, where $c \in (0, \infty)$ is an arbitrary but fixed positive real number. 
\end{proposition}
We may also take  $c=0$ or $c= \infty$ in \eqref{gk-smallrank2} provided that the 
integrals are all convergent.

\begin{proof}
We approximate $X$ by $A$ with distinct eigenvalues
$a_1 < \cdots < a_{\nu} < x_1 < \cdots x_n$ where the $a_j$ are close to $0$.
Then $B = A + v v^*$ has eigenvalues 
$b_1 < \cdots < b_{\nu-1} < y_1 < \cdots < y_{n+1}$ that interlace with those
of $A$, with a joint density, see \eqref{y-density6},
\[  \prod_{j=1}^{\nu-1} e^{-b_j} \prod_{j=1}^{n+1} e^{-y_j} \prod_{j=1}^{\nu} e^{a_j} \prod_{j=1}^n e^{x_j}
	\frac{\Delta_{\nu-1}(b)  \left(\prod_{j=1}^{\nu-1} \prod_{k=1}^{n+1} (y_k-b_j) \right)
	 \Delta_{n+1}(y)}{\Delta_{\nu}(a) \left(\prod_{j=1}^{\nu} \prod_{k=1}^n (x_k - a_j) \right) \Delta_{n}(x)}
	\]
subject to the interlacing conditions.
		
We restrict this to the $y$-variables by integrating out $b_1, \ldots, b_{\nu-1}$. This gives the
joint density for $y_1, \ldots, y_{n+1}$
\begin{multline*} 
	\prod_{j=1}^{n+1} e^{-y_j} \prod_{j=1}^n e^{x_j} \prod_{j=1}^{\nu} e^{a_j} 
	\frac{\Delta_{n+1}(y)}{\Delta_{\nu}(a) \left(\prod_{j=1}^{\nu} \prod_{k=1}^n (x_k - a_j) \right) \Delta_n(x)}
	\\
\int_{a_1}^{a_2} db_1 \cdots \int_{a_{\nu-1}}^{a_{\nu}} d b_{\nu-1} 
	\prod_{j=1}^{\nu-1} e^{-b_j} \Delta_{\nu-1}(b)   \left(\prod_{j=1}^{\nu-1} \prod_{k=1}^{n+1} (y_k-b_j) \right).
	 \end{multline*}
In the limit where all $a_j \to 0$ we also have that all $b_j \to 0$
because of the interlacing. Then $A \to X$,  $B \to Y$, and  using also \eqref{ab-integral} we find
the limiting joint density for the nonzero eigenvalues  $y_1, \ldots, y_{n+1}$ of $Y$
\[ \frac{1}{Z_{n+1}} \prod_{j=1}^{n+1} \left(y_j^{\nu-1} e^{-y_j} \right) \,
	\prod_{j=1}^n \left( x_j^{-\nu} e^{x_j} \right) \, \frac{\Delta_{n+1}(y)}{\Delta_n(x)}
	 \]
subject to the interlacing $0 < y_1 < x_1 < \cdots < x_n < y_{n+1}$.
The interlacing is encoded by the determinant $\det \left[ \chi_{y_j < x_k < y_{n+1}} \right]_{j,k=1}^n$.

Next, averaging over the polynomial ensemble \eqref{f-pol-ensemble} and using the
Andreief identity \eqref{Andreief-identity}, we find in a now 
familiar fashion a joint density
\begin{equation} \label{y-density8} 
	\frac{1}{\tilde{Z}_{n+1}} \Delta_{n+1}(y) \,
		\det\left[ y_j^{\nu-1} e^{-y_j} \int_0^{\infty} x^{-\nu} e^x f_k(x) \chi_{y_j < x < y_{n+1}} dx \right]_{j,k=1}^n
	\end{equation}
	for the eigenvalues of $Y$.
The $n \times n$ determinant in \eqref{y-density8} is extended to an $(n+1) \times (n+1)$ determinant
by adding first a row with zeros and then a column with ones. Then after  elementary
column operations we easily arrive at the polynomial ensemble  \eqref{g-pol-ensemble8} with functions \eqref{gk-smallrank2}
and a possibly different constant $\tilde{Z}_{n+1}$.
\end{proof} 

The two Propositions \ref{SmallRankModification} and \ref{SmallRankModification2}
have the following consequences regarding squared singular values of an extension
of a matrix by one row or one column.

\begin{corollary} \label{Laguerrecor1} Suppose $\nu \geq 0$.
Let $X$ be an $(n + \nu) \times n$ random matrix with squared singular values $0 < x_1 < \cdots < x_n$
that form a polynomial ensemble \eqref{f-pol-ensemble} with certain functions $f_1, \ldots, f_n$ on $[0,\infty)$.
Let $Y = \begin{pmatrix} X \\ v^* \end{pmatrix}$ with $v$ a random vector of independent
complex Gaussians as in Proposition \ref{SmallRankModification},
which is independent of $X$.
Then the squared singular values $y_1, \ldots, y_n$ of $Y$ are a polynomial ensemble
\eqref{g-pol-ensemble7} with functions 
\begin{equation} \label{gk-smallrank3} 
	g_k(y) = \int_0^y e^{-x} f_k(y-x) \, dx.
	\end{equation}
\end{corollary}

\begin{proof}
The squared singular values of $X$ are the eigenvalues of $X^*X$.
The squared singular values of $Y$ are the eigenvalues of 
$\begin{pmatrix} X^* & v \end{pmatrix} \begin{pmatrix} X \\ v^* \end{pmatrix} = X^* X + v v^*$. 
Thus the result follows from Proposition \ref{SmallRankModification}. 
The integration in \eqref{gk-smallrank3} extends to $y$ only, 
and not to $\infty$ as in \eqref{gk-smallrank}, since $f_k(x)$ is defined for $x \geq 0$ only,
and  we consider $f_k(y-x)$ to be zero if $x > y$.

\end{proof}

\begin{corollary} \label{Laguerrecor2} Suppose $\nu \geq 1$.
Let $X$ be an $(n + \nu) \times n$ matrix with squared singular values $0 < x_1 < \cdots < x_n$
that are a polynomial ensemble \eqref{f-pol-ensemble} with certain functions $f_1, \ldots, f_n$ on $[0,\infty)$.
Let $Y = \begin{pmatrix} X & v \end{pmatrix}$ with $v$ a random vector of independent complex Gaussians as in Proposition \ref{SmallRankModification2}, which is independent of $X$.
Then the squared singular values $y_1, \ldots, y_{n+1}$ of $Y$ are a polynomial ensemble
\eqref{g-pol-ensemble8} with functions \eqref{gk-smallrank3}.
\end{corollary}

\begin{proof}
The squared singular values of $X$ are the non-zero eigenvalues of $XX^*$, and
the squared singular values of $Y$ are the non-zero eigenvalues of $Y Y^*= X X^* + v v^*$.
Thus the result follows from Proposition \ref{SmallRankModification2}.
\end{proof}

It is interesting to note that a combination of Corollaries \ref{Laguerrecor1} and \ref{Laguerrecor2}
leads to the proof of one of the classical results of random matrix theory \cite{AGZ,For},
namely that the squared singular values of a complex Ginibre matrix are distributed
as a Laguerre ensemble. See also \cite[Chapter 4.3.3]{For} for a similar
approach, and \cite{For2} for related results.

\begin{corollary} \label{Laguerrecor3} 
Suppose $X$ is an $m \times n$ random matrix such that
$\Re X_{i,j}$, $\Im X_{i,j}$, $i= 1, \ldots, m$, $j=1, \ldots,n$ are independent
normal random variables with mean zero and variance $1/2$. Suppose $\nu = m-n \geq 0$. Then the
squared singular values $x_1, \ldots, x_n$ of $X$ have the joint density
\begin{equation} \label{LaguerreUE} 
	\frac{1}{Z_n} \, \prod_{i < j} (x_j-x_i)^2 \, \prod_{j=1}^n  x_j^{\nu} e^{-x_j}, \qquad \text{all } x_j > 0. 
	\end{equation}
\end{corollary}

\begin{proof} We use induction.
It is easy to check Corollary \ref{Laguerrecor3} for $m=n=1$.

Assume Corollary \ref{Laguerrecor3} holds for certain $m, n \geq 1$.
Note that \eqref{LaguerreUE} is a polynomial ensemble with functions $f_k(x) = x^{\nu+k-1} e^{-x}$ for $k=1, \ldots, n$.
Then by Corollary \ref{Laguerrecor1} it will follow that Corollary \ref{Laguerrecor3} also holds for $m+1$ and $n$,
and by Corollary \ref{Laguerrecor2} it holds for $m$ and $n+1$, provided that $m > n$. 
The calculations are straightforward and we do not give them explictly here.
\end{proof}

\section{Matrix extensions} \label{sec:extension}

In this final section we start from an $n \times n$ Hermitian matrix $X$ and we are going to extend
it to an $(n+1) \times (n+1)$ matrix by adding one row and one column. We write
\begin{equation} \label{Y-extension} 
	Y = \begin{pmatrix} X & v \\ v^* & c \end{pmatrix} 
	\end{equation}
where $v = \begin{pmatrix} v_1 & v_2 & \cdots & v_n \end{pmatrix}^t$ is a complex column vector, and $c \in \mathbb R$ is real.

The following result was given by Forrester \cite{For2}
and Adler, Van Moerbeke and Wang \cite{AvMW},
see also  \cite[Chapter 4.3.2]{For} and  \cite[section 3.1]{FN}, 
where the focus is on the situation where
$X$ is an $n \times n$ GUE matrix.

\begin{theorem} \label{HermitianExtensionfixed}
Suppose $c$, $\Re v_j$, and $\Im v_j$ for $j=1, \ldots, n$ are  independent 
normal random variables with mean zero, where $c$ has variance $1$ and $\Re v_j$, $\Im v_j$ 
have variance $1/2$. Assume $X$ has simple eigenvalues 
$x_1 < x_2 < \cdots < x_n$.
Then with probability one, the ordered eigenvalues $y_1 \leq y_2 \leq \cdots \leq y_{n+1}$ of $Y$ are simple, and strictly interlace with those of $X$:
\begin{equation} \label{yx-interlacing}
	y_1 < x_1 < y_2 < x_2  < \cdots < y_n < x_n < y_{n+1}. \end{equation}
In addition, the eigenvalues of $Y$ have the probability density 
\begin{equation} \label{y-density9}
\frac{1}{\sqrt{2\pi}} \, 
	\prod_{j=1}^{n+1} e^{-\frac{1}{2} y_j^2}   \prod_{j=1}^n e^{\frac{1}{2} x_j^2} \, \frac{\Delta_{n+1}(y)}{\Delta_n(x)} 
	\end{equation}
on the subset of $\mathbb R^{n+1}$ defined by the interlacing condition \eqref{yx-interlacing}.
\end{theorem}

\begin{proof}
See Lemma 1 in \cite{AvMW} or Proposition 6 in \cite{For2}.
\end{proof}

As before, there is an immediate consequence of Theorem
\ref{HermitianExtensionfixed} to polynomial ensembles.

\begin{proposition} \label{HermitianExtension}
Suppose $c$, $\Re v_j$, and $\Im v_j$ for $j=1, \ldots, n$ are mutually independent 
normal random variables with mean zero, where $c$ has variance $1$ and $\Re v_j$, $\Im v_j$ 
have variance $1/2$.   
Suppose that $X$ is a random Hermitian matrix of size $n \times n$, independent of $c$ and $v$, whose 
eigenvalues are a polynomial ensemble \eqref{f-pol-ensemble} with certain functions $f_1, \ldots, f_n$.
Then the eigenvalues $y_1, \ldots, y_{n+1}$ of $Y$ given by \eqref{Y-extension} are a polynomial ensemble
\begin{equation} \label{g-pol-ensemble6} 
	\frac{1}{\tilde{Z}_{n+1}} \, \Delta_{n+1}(y) \, \det \left[ g_{k}(y_j) \right]_{j,k=1}^{n+1}, 
	\end{equation}
with functions
\begin{equation}  \label{gk-extension} 
\begin{aligned} 	
	g_{1}(y) & = e^{-\frac{1}{2} y^2},  \\
	g_{k+1}(y) & = e^{-\frac{1}{2} y^2} \int_{0}^y e^{\frac{1}{2} x^2} f_k(x) \, dx, \qquad k=1, \ldots, n, \\
		\end{aligned}
		\end{equation}
\end{proposition}

\begin{proof}
The eigenvalues of $X$ are distinct with probability one. We order them, say $x_1 < x_2 < \cdots < x_n$.

We use an interlacing determinant as in \eqref{interlacing-det} to write
the density \eqref{y-density9} as
\[ \frac{1}{Z_{n+1}}  \prod_{j=1}^{n+1} e^{-\frac{1}{2} y_j^2}  \prod_{j=1}^n e^{\frac{1}{2} x_j^2}
	\,\frac{\Delta_{n+1}(y)}{\Delta_n(x)} 
	\det \left[ \chi_{y_j \leq x_k} \right]_{j,k=1}^{n+1} \]
with $x_{n+1} = +\infty$ and a certain constant $Z_{n+1}$, which is also
\begin{equation} \label{y-density10}  
	\frac{1}{Z_{n+1}}
	\prod_{j=1}^{n+1} e^{-\frac{1}{2} y_j^2}  \prod_{j=1}^n e^{\frac{1}{2} x_j^2}
	\, \frac{\Delta_{n+1}(y)}{\Delta_n(x)} 
	\det \left[ \chi_{y_j \leq x_k < y_{n+1}} \right]_{j,k=1}^{n}. \end{equation}
Then averaging \eqref{y-density10} with respect to the polynomial ensemble \eqref{f-pol-ensemble}
and using the Andreief identity \eqref{Andreief-identity} we obtain for the density of $y_1, \ldots, y_{n+1}$,
\[ \frac{1}{\tilde{Z}_{n+1}} \prod_{j=1}^{n+1} e^{-\frac{1}{2} y_j^2} \Delta_{n+1}(y) 
	\det \left[ \int_{y_j}^{y_{n+1}} e^{\frac{1}{2} x^2} f_k(x) dx \right]_{j,k=1}^n. \]
We extend the last determinant to an $(n+1) \times (n+1)$ determinant by adding 
first a row with entries  $ \int_{y_{n+1}}^0  e^{\frac{1}{2}x^2} f_k(x) dx$ for $k=1, \ldots, n$ and
then a column $\begin{pmatrix} 0 & \cdots & 0 & 1 \end{pmatrix}^t$ of length $n+1$.
Then we subtract the last row from each of the other rows and change the sign in
each of the entries in rows $1$ up to $n$. Then the result is \eqref{g-pol-ensemble6}
with functions \eqref{gk-extension} and a possibly different constant $\tilde{Z}_{n+1}$.
\end{proof}

As a special case, we consider the polynomial ensemble \eqref{f-pol-ensemble} with
functions $f_k(x) = x^{k-1} e^{-\frac{1}{2} x^2}$ for $k=1, \ldots, n$. This is the same as
\begin{equation} \label{GUE-density} 
	\frac{1}{Z_n} \, \Delta_n(x)^2 \, \prod_{j=1}^n e^{-\frac{1}{2} x_j^2}. 
	\end{equation}
Then by \eqref{gk-extension} we get $g_1(y) = e^{-\frac{1}{2}y^2}$ and
\[ g_{k+1}(y) = e^{-\frac{1}{2}y^2} \int_0^y x^{k-1} dx =  \frac{1}{k} y^k e^{-\frac{1}{2} y^2}
	\qquad \text{for } k =1, \ldots, n. \]
The prefactor $\frac{1}{k}$ is immaterial and it follows from Proposition \ref{HermitianExtension}
that the density function for the  eigenvalues of $Y$ is
\[ \frac{1}{Z_{n+1}} \, \Delta_{n+1}(y)^2 \, \prod_{j=1}^{n+1} e^{-\frac{1}{2} y_j^2}. \]
It is well-known that \eqref{GUE-density} is the density of eigenvalues of GUE random matrix
\cite{AGZ,For} and we conclude, as already noted in \cite[Chapter 4.3.2]{For},
 that we can use Proposition \ref{HermitianExtension} to give an
inductive proof of this basic result of random matrix theory.

\subsubsection*{Acknowledgements}

The author thanks Peter Forrester for useful correspondence and 
for pointing out relevant references to the literature.

The author is supported by KU Leuven Research Grant OT/12/073, the Belgian 
Interuniversity Attraction Pole P07/18, and FWO Flanders projects G.0641.11 and
G.0934.13.

\end{document}